\documentclass[letterpaper, 10 pt, conference]{ieeeconf}
                             
\IEEEoverridecommandlockouts                              
\usepackage{graphics}
\usepackage{epsfig} 
\usepackage{amsmath} 
\usepackage{amssymb} 
\usepackage{bm}
\usepackage{xcolor}
\usepackage{amsthm}
\usepackage{epsfig}
\usepackage{accents}
\usepackage{breqn}
\usepackage[ruled,vlined]{algorithm2e}
\usepackage{cite}
\usepackage{hyperref}

\theoremstyle{plain}
\newtheorem{theorem}{Theorem}
\newtheorem{lemma}{Lemma}
\newtheorem{definition}{Definition}

\newtheorem{problem}{Problem}
\newtheorem{remark}{Remark}
\newtheorem{example}{Example}

\theoremstyle{definition}

\newcommand{\half}{\frac{1}{2}}
\newcommand{\eqn}[1]{\begin{align}#1\end{align}}
\newcommand{\eqnN}[1]{\begin{align*}#1\end{align*}}
\newcommand{\deqn}[1]{\begin{dmath}#1\end{dmath}}
\renewcommand{\to}{\rightarrow}
\newcommand{\argmin}{\operatorname*{argmin}}
\newcommand{\minimise}{\operatorname*{minimise}}

\newcommand{\bmat}[1]{\begin{bmatrix}#1\end{bmatrix}}

\renewcommand{\S}{\mathcal{S}}
\newcommand{\C}{\mathcal{C}}

\newcommand{\X}{\mathcal{X}}
\newcommand{\U}{\mathcal{U}}
\newcommand{\K}{\mathcal{K}}

\newcommand{\Int}{\text{Int}}
\newcommand{\reals}{\mathbb{R}}
\newcommand{\classK}{class-$\K$ }
\newcommand{\classKinf}{\mbox{class-$\K_\infty$} }

\newcommand{\KICCBF}{K_\text{\scalebox{.6}{ICCBF}}}

\title{\LARGE \bf
Safe Control Synthesis via Input Constrained Control Barrier Functions
}

\author{Devansh R. Agrawal and Dimitra Panagou
\thanks{The authors would like to acknowledge the support of the National Science Foundation under award number 1942907.}
\thanks{Both authors are with the Department of Aerospace Engineering, University of Michigan, Ann Arbor, MI, USA; \texttt{\{devansh, dpanagou\}@umich.edu}}%
}

\begin{document}

\maketitle
\thispagestyle{empty}
\pagestyle{empty}

\begin{abstract}
This paper introduces the notion of an Input Constrained Control Barrier Function (ICCBF), as a method to synthesize safety-critical controllers for non-linear control affine systems with input constraints. The method identifies a subset of the safe set of states, and constructs a controller to render the subset forward invariant. The feedback controller is represented as the solution to a quadratic program, which can be solved efficiently for real-time implementation. Furthermore, we show that ICCBFs are a generalization of Higher Order Control Barrier Functions, and thus are applicable to systems of non-uniform relative degree. Simulation results are presented for the adaptive cruise control problem, and a spacecraft rendezvous problem.
\end{abstract}

\section{Introduction}

Many cyber-physical systems are safety critical, that is, they require guarantees that safety constraints are not violated during operation. Safety is often modeled by defining a safe subset of the state space for a given system, within which the state trajectories must always evolve. Recently, set-theoretic methods, such as Control Barrier Functions (CBFs) have become increasingly popular as a means of constructing and verifying such controllers~\cite{Ames2019, amesTAC2017, ASIF, kunalCBF}. 

Quadratic Programs (QPs) enforcing CBF conditions have become popular for synthesizing low-level controllers to ensure that system trajectories remain within the safe set. These QPs can be solved in real-time at the current state, enabling safety-critical controllers for a wide range of non-linear systems to be computed efficiently during operation.

Prior work on CBFs has largely focused on systems where a sufficiently large control authority is available to ensure forward invariance of the safe set. However in the presence of input constraints, only a subset of the safe set may be rendered forward invariant. A few methods have been proposed to compute such subsets of the safe set. These include reachability analysis by solving a Hamilton-Jacobi partial differential equations over the state-space~\cite{levelset, ASIF} and Sum-of-Squares (SOS) based semi-definite programs, which employ the positivstellensatz theorem to provide a certificate of safety~\cite{parilloSOS, SOSQP}. Both methods scale poorly with the dimension of the state-space, and SOS methods assume the dynamics and safety constraints can be expressed as polynomials. Some methods have also been proposed for dynamical systems of specific classes, for instance Euler-Lagrange systems~\cite{ELsystems}.

In this paper, we introduce the notion of an Input Constrained Control Barrier Function (ICCBF). An ICCBF guarantees that an input constrained controller can render a subset of the safe set forward invariant. We present the construction of these functions, and define the subset of the safe set, and a corresponding controller, such that the subset is guaranteed to be forward invariant under this controller. Furthermore, we show that ICCBFs represent a generalization of Higher Order CBFs (HOCBFs)~\cite{hocbf}, enabling synthesis of input-constrained controllers for safe sets of higher, and non-uniform relative degree. Finally, the method is applied to an adaptive cruise control problem~\cite{ames2014},  and a spacecraft rendezvous problem, demonstrating that ICCBFs define a safe controller that respects input constraints.

The paper is structured as follows: In Section 2, some preliminaries on set invariance, as well as the problem statement are presented. Section 3 presents the definition of ICCBFs, and the synthesis of corresponding feedback controllers for ensuring safety. Finally, simulation results are presented in section 4.

\section{Problem Formulation and Preliminaries}

\subsection{Notation}
We denote the set of real numbers $\reals$ and non-negative reals $\reals_+$. A continuous function $\alpha: [0, a) \rightarrow [0, \infty)$ is \classK if it is strictly increasing on the domain, and $\alpha(0) = 0$. It is \classKinf if $a = \infty$ and $\lim_{r\to\infty} \alpha(r) = \infty$. A continuous function $\alpha: (-b, a) \rightarrow (-\infty, \infty)$ is an extended \classK function, if for some $a, b > 0$, if it is strictly increasing and $\alpha(0) = 0$. The Lie derivative of $h(x)$ along $f(x)$ is denoted $L_f h(x) = \frac{d h}{d x} f(x) $. The short hand $\dot h(x) = \frac{d h}{dx} \dot x$ will also be used to indicate time derivatives of $h$ along flows from a state $x$. We denote $\Int(\C)$ and $\partial \C$ as the interior and boundary of a set $\C$. The empty set is denoted $\emptyset$. 

\subsection{Problem Setup}
Consider a nonlinear, control-affine dynamical system, with state $x \in \X \subset \reals^n$ and control input $u \in \U \subset \reals^m$ given as
\eqn{
\dot x = f(x) + g(x) u
\label{eqn:sys}}
where $f : \X \rightarrow \reals^n$, $g: \X \rightarrow \reals^{n \times m}$ are sufficiently smooth, as discussed in \ref{sec:ICCBFs}. We will assume these functions are known, and the system state is available to the controller. 

We define a state $x$ as \emph{safe} if it lies in a set $\S$, which is defined as the 0-superlevel set of a continuously differentiable function $h : \X \rightarrow \reals$, as follows:
\eqn{
\S &\triangleq \{ x \in \X : h(x) \geq 0\} \label{eqn:setDef}\\
\partial \S &\triangleq \{ x \in \X : h(x) = 0\} \label{eqn:setDefdS}\\
\Int(\S) &\triangleq \{ x \in \X : h(x) > 0 \label{eqn:setDefIntS}\}
}
The set $\S$ is referred to as the \emph{safe} set. We will assume this set is closed, non-empty and is simply connected.

\begin{definition}
Under a control input $u \in \U$, the set $\S$ is \emph{forward invariant} for the closed-loop system~(\ref{eqn:sys}) if for all $x(0) \in \S$, it holds that $x(t) \in \S$ for all $t\geq 0$.
\end{definition}

Under input constraints, however, it may be impossible to render the safe set forward invariant, as in Example \ref{eg:1}:
\begin{example}
\label{eg:1}
Consider the unstable single-input, scalar dynamical system with input constraints:
\eqnN{\dot x = x + u, \quad u \in \U = [-1, 1]}
Let the safety set be \eqnN{\S = \{ x\in \reals : x \leq 2 \}, \quad h(x) = 2 -x} Now consider the boundary state $x = 2$. In this case,
\eqnN{\dot h(x, u) = -2 - u \leq -1}
and therefore under any $u\in \U$, closed-loop trajectories will leave safe set. Thus $\S$ cannot be rendered forward invariant under the input constraints. \hfill $\triangle$
\end{example}

To address the problem when the safe set cannot be rendered forward invariant, we define an \emph{inner safe set} as follows:
\begin{definition}
A non-empty closed set $\C^*$ is an \emph{inner safe set} of the safe set $\S$ for the dynamical system~(\ref{eqn:sys}), if $\C^* \subseteq \S$ and there exists a feedback controller $u : \C^* \rightarrow \U$ such that for the closed-loop system~(\ref{eqn:sys}), for all $x(0) \in \C^*$, it holds that $x(t) \in \C^*$ for all $t\geq 0$. 
\end{definition}
Now, we can state the main objective of this paper:
\begin{problem}
Given the system (\ref{eqn:sys}), find a closed set $\C^* \subseteq \S$ and a feedback controller $u : \C^* \rightarrow \U$, such that for any $x(0) \in \C^*$, the closed-loop trajectories of~(\ref{eqn:sys}) satisfy $x(t) \in \C^*$ for all $t \geq 0$.
\end{problem}

In words, the objective is to find a subset of the safe set, and a corresponding feedback controller that renders the subset forward invariant.

\subsection{Set Invariance}

Nagumo's theorem provides a necessary and sufficient condition for the forward invariance of a set $\S$. 
\begin{lemma} \emph{\cite[Thm.~3.1]{blanchini}}
Consider the system $\dot x = f(x)$, and assume that for each initial condition $x\in \X$, it admits a globally unique solution. Then, a closed set $\S \subseteq \X$ is forward invariant for the system, if and only if 
\deqn{
f(x) \in \hat{T}_{\S}(x)\condition{$\forall x \in \partial \S$},
}
where $\hat{T}_{\S}(x) $ is the tangent cone of $\S$ at $x$. 
\end{lemma}

In this work, since $\S$ is defined by $h(x) \geq 0$, Nagumo's theorem simplifies to the following statement.
\begin{lemma}
Consider the system~(\ref{eqn:sys}). Let the set $\S$ be defined by a continuously differentiable function $h : \X \rightarrow \reals$, as per (\ref{eqn:setDef}-\ref{eqn:setDefIntS}). Consider a controller $u : \S \rightarrow \U$, such that for any initial condition $x \in \S$, the closed-loop system~(\ref{eqn:sys}) admits a globally unique solution. Then set $\S$ is forward invariant if and only if 
\deqn{
L_fh(x) + L_gh(x) u(x) \geq 0 \condition{$\forall x \in \partial \S$}.\label{eqn:nagumo}}
\end{lemma}

If $h$ is a Zeroing Control Barrier Function (ZCBF), a controller that renders $\S$ forward invariant is guaranteed to exist~\cite{amesTAC2017, Ames2019}. The focus of this paper is on scenarios where $h$ is not a ZCBF. In the following section, we propose Input Constrained Control Barrier Functions (ICCBFs) to address these cases. 

\section{Input Constrained Control Barrier Functions}

In this section we define Input Constrained Control Barrier Functions (ICCBFs). A brief motivation for the construction is provided, followed by the formal definitions. Finally a method to find ICCBFs is proposed. 

\subsection{Motivation}

\begin{figure}
    \centering
    \includegraphics[width=0.8\linewidth]{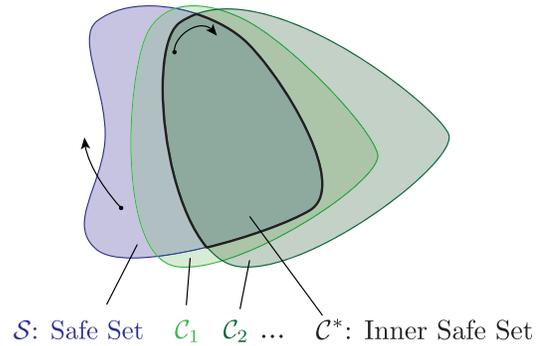}
    \caption{Visual representation of ICCBF method. The safe set $\S$ and two intermediate sets $\C_1$ and $\C_2$ are drawn. The final inner safe set $\C^*$ is the intersection of each of these sets, and can be rendered forward invariant. }
    \label{fig:SandC}
\end{figure}

Suppose the safe set $\S$ associated with $h$ cannot be rendered forward invariant by any controller $u\in \U$. Then we define a function $b_1 : \X \rightarrow \reals$ and a set $\C_1$ (visualised in Figure~\ref{fig:SandC}) as follows
\eqn{
b_1(x) &= \inf_{u\in \U} [\dot h(x, u) + \alpha_0(h(x))]\label{eqn:motivatingb1}\\
\C_1 &= \{ x \in \X : b_1(x) \geq 0\} 
}
where $\alpha_0$ is some user specified extended \classK function and the shorthand $\dot h(x, u) = L_fh(x) + L_gh(x) u$ is used. Notice that since an infimum over $\U$ is taken, $b_1$ is only dependent on the state $x$, and not the input $u$. 

The set $\C_1$ has a useful property: Suppose there exists a point $x\in \X$  such that $x \in \partial \S$ and $x \in \C_1$, i.e., $h(x) = 0$ and $b_1(x) \geq 0$. Then, from~(\ref{eqn:motivatingb1}) it follows, $\inf_{u\in \U} [\dot h] \geq 0$, and therefore $\dot h \geq 0$ \emph{for any} $u \in \U$, i.e.,
\eqn{
x \in \partial \S \cap \C_1 \implies \dot h(x, u) \geq 0, \forall u \in \U.\label{eqn:motivatingProperty}
}
Thus the closed-loop vector field $f(x) + g(x) u$ lies in the tangent cone of $\S$ at $x$, i.e., no closed-loop trajectory can leave $\S$ through such an $x$.

Now suppose there exists a Lipschitz continuous controller $u\in \U$ such that $\C_1$ is forward invariant. Then the set $\S \cap \C_1$ is also forward invariant: Consider an initial condition $x(0) \in \S \cap \C_1$. Suppose there exists a $t >0$ where $h(x(t)) < 0$. By continuity, there must be some $\tau \in [0, t)$ where $h(x(\tau)) = 0$ and $\dot h(x(\tau), u) < 0$. However, since the controller renders $\C_1$ forward invariant, we have $x(\tau) \in \partial S$ and $x(\tau) \in \C_1$. By~(\ref{eqn:motivatingProperty}), $\dot h(x(\tau), u) \geq 0$, which is a contradiction. Therefore, trajectories cannot reach a state $x \not\in \S$. 

To summarise, in systems where $\S$ cannot be rendered forward invariant, the subset set ${\S \cap \C_1}$ is forward invariant under any controller that renders $\C_1$ forward invariant.

In defining $b_1$, we allowed the user to specify the extended \classK function $\alpha_0$. However not all \classK functions will admit a controller that renders $\C_1$ forward invariant. Either a different \classK function could be used, or we can repeat the same steps: define $b_2(x) = \inf_{u \in \U} [\dot b_1(x, u) + \alpha_1(b_1(x))]$ and $\C_2 = \{ x \in \X: b_2(x) \geq 0\}$. Now any controller $u$ that renders $\C_2$ forward invariant also renders $\C_1 \cap \C_2$ forward invariant, and therefore the set $\S \cap \C_1 \cap \C_2$ is also forward invariant by the same controller. This idea is formalised in the next subsection.

\subsection{ICCBFs}
\label{sec:ICCBFs}

Consider the dynamical system~(\ref{eqn:sys}) with bounded control inputs $u\in \U$ and a safe set $\S$ defined by a function $h: \X \rightarrow \reals$, as per~(\ref{eqn:setDef}-\ref{eqn:setDefIntS}). Assume $h$ is not a ZCBF. We define the following sequence of functions:
\eqn{
b_0(x) &= h(x) \label{eqn:b0def}\\
b_1(x) &= \inf_{u\in \U} [L_fb_0(x) + L_gb_0(x) u + \alpha_0(b_0(x))] \\
b_2(x) &= \inf_{u\in \U} [L_fb_1(x) + L_gb_1(x) u + \alpha_1(b_1(x))] \label{eqn:b1def}\\
\vdots \nonumber\\
b_{i+1}(x) &= \inf_{u\in \U} [L_fb_i(x) + L_gb_i(x) u + \alpha_i(b_i(x))] \label{eqn:bidef}\\
\vdots \nonumber\\
b_N(x) &= \inf_{u\in \U} [L_fb_{N-1}(x) + L_gb_{N-1}(x) u \nonumber\\&{}\quad{}\quad{}\quad{}\quad + \alpha_{N-1}(b_{N-1}(x))] \label{eqn:bNdef}}
where each $\alpha_i$ is some user-specified \classK function, and $N$ is a positive integer.  We assume the functions $f, g, h$ are sufficiently smooth such that $b_N$ and its derivative are defined. Each function $b_i : \X \rightarrow \reals$ is a scalar function that only depends on the state. The time derivative $\dot b_i = \frac{db_i}{dx} \dot x =  L_fb_i(x) + L_gb_i(x) u$ is still affine in $u$. 

\begin{remark}
\label{remark:HOCBFcomp}
Higher Order CBFs, as in~\cite{hocbf}, are a special case of ICCBFs. For instance, in systems of relative degree 2, $L_gh(x) = 0$ for all $x \in \S$. In this case, in the construction of ICCBFs we have
\eqn{
b_1(x) &= \inf_{u\in \U} [ L_fh(x) + L_gh(x) u +\alpha_0(h(x))]\\
&= \inf_{u\in \U} [ L_fh(x)  +\alpha_0(h(x))]\\
&= L_fh(x)  + \alpha_0(h(x))}
which is exactly the function defined in~\cite{hocbf}. This repeats for any relative degree greater than~2, and thus for a system with relative degree $\rho$, the first $\rho$ expressions of ICCBFs are identical to those of HOCBFs. Moreover, ICCBFs can handle systems with non-uniform relative degree, by choosing $N$ greater or equal to the largest relative degree of the system in $\S$. 
\end{remark}

Next, we define a family of sets, 
\eqn{
\C_0 &= \{ x \in \X : b_0(x) \geq 0\} = \S \label{eqn:C_0} \\
\C_1 &= \{ x \in \X : b_1(x) \geq 0\} \label{eqn:C_1}\\
&\vdots \nonumber\\
\C_{i} &= \{ x \in \X : b_{i}(x) \geq 0\} \label{eqn:C_i}\\
&\vdots \nonumber\\
\C_N &= \{ x \in \X : b_N(x) \geq 0\}\label{eqn:C_N}
}
and the intersection of these sets 
\eqn{
\C^* = \C_0 \cap \C_1 \cap ... \cap \C_N \label{eqn:C}}
We assume the set $\C^*$ is closed, non-empty and has no isolated points.

\begin{definition}
For the above construction, if there exists a \classK function $\alpha_N$ such that
\eqn{
\sup_{u\in \U}[L_fb_N(x) + L_gb_N(x) u + \alpha_N(b_N(x))] \geq 0 \quad \forall x\in \C^*,\label{eqn:ICCBFdef}}
then $b_N$ is an \emph{Input Constrained Control Barrier Function (ICCBF)}. 
\end{definition}

Note, this does not require $b_N$ to be a ZCBF on $\C_N$. The definition only requires condition~(\ref{eqn:ICCBFdef}) to hold for all $x \in \C^*$ which is a subset of $\C_N$. The main result of this paper can now be stated:
\begin{theorem}[Main Result]
\label{theorem:main}
Given the input constrained dynamical system (\ref{eqn:sys}), if $b_N$, defined by (\ref{eqn:b0def} - \ref{eqn:bNdef}), is an ICCBF, then any Lipschitz continuous controller $u : \C^* \rightarrow \U$ such that $u(x) \in \KICCBF(x)$, where
\eqn{
\KICCBF(x) &=\{ u \in \U: \nonumber\\
& L_fb_N(x) + L_gb_N(x) u \geq -\alpha_N(b_N(x)) \}\label{eqn:KICCBF}}
renders the set $\C^* \subseteq \S$~(\ref{eqn:C}) forward invariant.
\end{theorem}

\begin{proof}
Since $u$ is a Lipschitz continuous controller, the closed-loop system~(\ref{eqn:sys}) is also Lipschitz continuous. To show forward invariance of $\C^*$, we use Nagumo's theorem on the closed-loop system. In particular, we show
\eqn{
&[f(x) + g(x) u(x)] \in \hat{T}_{\C^*}(x), \nonumber \\
&\quad \forall u(x) \in \KICCBF(x), \text{ and } \forall x \in \partial\C^*,\label{eqn:tangentConeCond}
}
where $\hat{T}_{\C^*}(x)$ is the tangent cone of $\C^*$ at $x$. At each $x \in \partial \C^*$, we have $b_i(x) = 0$ for some set, possibly all, of $i$'s in $\{0, ..., N\}$. Let $\mathcal{I}(x) = \{ i : b_i(x) = 0\}$. Condition~(\ref{eqn:tangentConeCond}) is satisfied if, $\dot b_i(x, u(x)) \geq 0$, for all $i \in \mathcal{I}(x)$, for all $x \in \C^*$, provided $u(x) \in \KICCBF(x)$. In other words, we show that for all $x \in \C^*$,
\eqn{
b_i(x) = 0 \implies \dot b_i(x, u(x)) \geq 0, \forall u(x) \in \KICCBF(x)\label{eqn:positivity}
}
for all $i \in \{ 0, ..., N\}$.

\emph{Cases $i\in \{0, ..., N-1\}$}: Consider any $x \in \C^* \cap \partial\C_i$. Since $\C^* \subseteq \C_{i+1}$, $x\in\partial\C_i \cap \C_{i+1}$. 
By (\ref{eqn:bidef}, \ref{eqn:C_i}),
\eqnN{
\inf_{u\in\U} [L_fb_i(x) + L_gb_i(x)] &\geq 0\\
\therefore L_fb_i(x) + L_gb_i(x) &\geq 0, \quad \forall u\in \U
}
and therefore $\dot b_i(x, u) \geq 0$, for all $u \in \U$. Since $\KICCBF(x) \subseteq \U$, condition~(\ref{eqn:positivity}) is satisfied.

\emph{Case $i = N$}: Consider $x \in \C^* \cap \partial \C_N$. Since $b_N$ is an ICCBF and $b_N(x) = 0$, by~(\ref{eqn:C_N}, \ref{eqn:KICCBF}),
\eqnN{
L_fb_N(x) + L_gb_N(x) u(x) \geq 0, \quad \forall u(x) \in \KICCBF(x)
}
and therefore $\dot b_N(x, u) \geq 0$ for any $u(x) \in \KICCBF(x)$.

In conclusion, we have shown that condition~(\ref{eqn:positivity}) is satisfied for all $i$, and therefore the conditions of Nagumo's theorem are satisfied, which completes the proof.\end{proof}

\begin{remark}
The practical value of this construction is that for a given system, a set $\S$ of safe states of practical importance can be specified, which may not be rendered forward invariant under the given system dynamics. By using ICCBFs, we remove states of the set $\S$, and construct an inner set for which we can find a controller that renders it forward invariant.
\end{remark}

We would like to note a useful special case, the \emph{simple ICCBF}:
\begin{definition}
In the above construction, if $\C^*$ is a strict subset of $\C_N$, i.e., $\C^* \subset \C_N$, then $b_N$ is a \emph{simple ICCBF}.  \label{def:trivialICCBF}
\end{definition}

\begin{theorem}
\label{theorem:trivialICCBF}
For the dynamical system~(\ref{eqn:sys}), if $b_N$ is a simple ICCBF, all Lipschitz continuous controllers $u(x)\in \U$ render the set $\C^*$ forward invariant.
\end{theorem}
\begin{proof}
By definition, since $b_N$ is a simple ICCBF, $\C^*$ is a strict subset of $\C_N$. Then $\C^* \cap \partial \C_N = \emptyset$, i.e., there does not exist a $x\in\C^*$ such that $b_N(x) = 0$. Following Theorem~\ref{theorem:main}, we do not need to consider case where $i = N$ in condition~(\ref{eqn:positivity}). The remaining cases, with $i\in\{0, ..., N-1\}$ satisfy condition~(\ref{eqn:positivity}) for all $u(x) \in \KICCBF(x)$. Therefore, any Lipschitz continuous $u\in \U$ admits globally unique solutions and satisfies condition~(\ref{eqn:tangentConeCond}), completing the proof.
\end{proof}

\begin{remark}
Intuitively, the existence of a simple ICCBF represents a system where the dynamics at the boundaries of $\C^*$ are such that the unforced dynamics $f(x)$ dominate the forcing term $g(x) u$ in driving the system towards safety. If a simple ICCBF is found, no safety critical controller is needed for the system to ensure state trajectories remain within the safe set, provided the system is initialised within~$\C^*$. 
\end{remark}

\subsection{QP-based Control Synthesis}
\label{sec:ICCBF-QP}

Finally, similar to~\cite{amesTAC2017, ASIF}, we define an optimization-based controller that will render $\C^* \subseteq \S$ forward invariant. In this section, suppose the set $\U$ can be expressed as a convex polytope, i.e., $\U = \{ u \in \reals^m : A u \leq B\}$ where $A \in \reals^{p\times m}$ and $B \in \reals^p$ for some positive integer $p$. Then the a controller based on the solution to the following Quadratic Program (QP) renders the set $\C^*$ forward invariant.

\begin{theorem}
Consider the dynamical system~(\ref{eqn:sys}). Suppose $b_N$ is an ICCBF associated with the \classK function $\alpha_N$ and inner safe set $\C^*$. Let $H: \X \to \reals^{m\times m}_+$,  where $\reals^{m\times m}_+$ is the set of real $m\times m$ positive definite matrices, and $F: \X \to \reals^{m}$ be Lipschitz continuous cost functions. The solution to the following quadratic optimization problem
\eqn{
\begin{array}{rl}
u^*(x) = \underset{u \in \reals^m}{\argmin} & \half u^T H(x) u + F(x)^T u \\
\text{subject to} & L_fb_N(x) + L_gb_N(x) u \geq - \alpha_N(b_N(x))\\
& u \in \U
\end{array}\label{eqn:ICCBF-QP}
}
yields an input-constrained feedback controller $u^*: \C^* \rightarrow \U$ that renders $\C^*$ forward invariant. Thus, $h(x(t)) \geq 0$ for all $t \geq t_0$, i.e. the system trajectories always evolve within the safe set $\S$. 
\end{theorem}
\begin{proof}
The constraints in (\ref{eqn:ICCBF-QP}) specify that $u^*(x) \in \KICCBF(x)$. Since the $u^*(x)$ is a Lipschitz continuous controller, (following the same arguments as~\cite[Thm.~8]{robustnessXu}), by Theorem~\ref{theorem:main}, the controller $u^*(x)$ obtained as the solution to the optimization problem~(\ref{eqn:ICCBF-QP}) renders $\C^*$ forward invariant.
\end{proof}

\begin{remark}
Notice the final optimization-based controller only requires $b_N$ and its derivatives. This means the computational complexity of the controller is almost identical to the CBF-QP based controllers of~\cite{ames2014,amesTAC2017}.~\footnote{Computing $b_N$ and its derivative can involve long expressions. In our implementations, we used \texttt{Mathematica}~\cite{mathematica} for symbolic manipulation, for the ACC case study, and \texttt{Julia}'s auto-differentiation package \texttt{ForwardDiff.jl}~\cite{bezanson2017julia, forwardDiff}, for the autonomous rendezvous case study.}
\end{remark}

To conclude, we have introduced the notion of an ICCBF. Once an ICCBF is found, the inner safe set $\C^* \subseteq \S$ given by~(\ref{eqn:C}) and the quadratic-program based feedback controller $u^* : \C^* \to \U$ given by~(\ref{eqn:ICCBF-QP})  renders the $\C^*$ forward invariant. In this way, ICCBFs address Problem 1. 

\section{Simulations}

\subsection{Adaptive Cruise Control}

\begin{figure*}
    \centering
    \includegraphics[width=\linewidth]{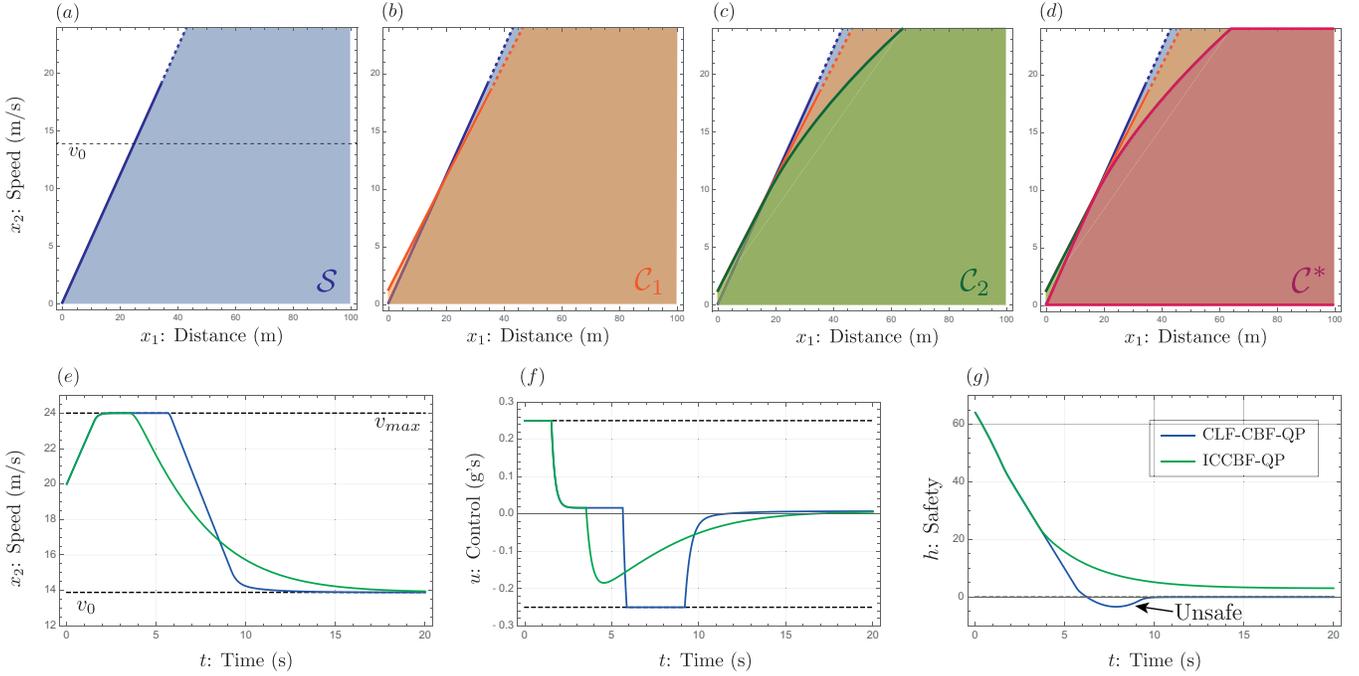}
    \caption{Figures (a-d): State-space diagrams indicating the sets (a)~$\S$, (b)~$\C_1$, (c)~$\C_2$ and (d)~$\C^*$. The horizontal dashed line in (a) indicates $v_0$, the speed of the car in-front. Figure~(d) represents the inner safe set $\C^*$ that is rendered forward invariant. Figures (e-g): Simulation results for speed, control input and safety under the CLF-CBF-QP controller~\cite{ames2014} and the ICCBF-QP.}
    \label{fig:sets_acc}
\end{figure*}

As a demonstration, we apply ICCBFs to the Adaptive Cruise Control (ACC) problem of~\cite{ames2014}. Consider a point-mass model of a vehicle moving in a straight line. The vehicle is following a vehicle in-front, which moves at a known constant speed $v_0$. The objective is to design a controller which will prevent the vehicle from colliding with the vehicle in-front, but will allow it to accelerate to the speed limit when it is safe to do so.

Using the model defined in~\cite{ames2014}, the safety constraint is specified as $d \geq 1.8 v$, where $d$ is the distance to the vehicle in-front, and $v$ is the speed of the vehicle.  Defining the state of the system $x = [d, v]^T$, the dynamical model is
\eqnN{
\bmat{\dot d\\ \dot v} = \bmat{v_0 - v\\ -F(v)/m} + \bmat{0\\g_0}u
}
where $F(v) = f_0 + f_1 v + f_2 v^2$ models resistive forces on the vehicle, $m$ is the mass of the vehicle, $g_0$ is acceleration due to the gravity. The control input $u$ is the acceleration of the vehicle, and we suppose it is constrained to
\eqnN{\U = \{ u : -0.25 \leq u \leq 0.25 \}}
which represents a maximum deceleration of $0.25g$'s. The safe set $\S$ is defined by the function $h$,
\eqnN{\S &= \{x \in \X : h(x) \geq 0\}\label{eqn:ACC_S}\\
h(x) &= x_1 - 1.8 x_2}
and we can verify that $\S$ is not forward invariant under the input constraints. Thus, $h$ is not a ZCBF, and we will apply ICCBFs to find an inner safe set. 

We choose, arbitrarily, $N = 2$ and the \classK functions 
\eqnN{
\alpha_0(h) = 4 h, \quad \alpha_1(h) = 7 \sqrt{h}, \quad \alpha_2(h) = 2h,
}
to define the functions $b_1, b_2$ and sets $\C_1, \C_2$.

The sets are visualized in Figure~\ref{fig:sets_acc}. The interior of a set is shaded, and the boundary of the set is indicated with a thick line. Each boundary is partitioned into two subsets: the subset where there exists a feasible control input to keep trajectories within the set are indicated with a solid line, and the subset where no feasible control input will keep trajectories within the set are indicated with a dashed line.

To verify that $b_2$ is an ICCBF, the following optimization problem was used:
\eqnN{
\begin{array}{rl}
\gamma = \underset{x \in \X}{\minimise}  & \sup_{u\in\U} [\dot b_N(x, u) + \alpha_N(b_N(x))] \\
\text{subject to} & x \in \C^*
\end{array}
}

By the definition of ICCBFs, $b_N$ is an ICCBF if and only if the optimization problem is feasible, with solution $\gamma \geq 0$.  While this optimization problem is, in general, non-linear and non-convex, a standard non-linear optimization software can be used to solve it. Applied to the ACC problem, we determined $\gamma = 2.33 \geq 0$, and therefore $b_2$ is a valid ICCBF.

Finally, we define $\C^*$ as the intersection of $\S, \C_1, \C_2$. This set is visualised in Figure~\ref{fig:sets_acc}(d). The following controller renders the set forward invariant:
\eqnN{
\begin{array}{rl}
u^*(x) = \underset{u \in \reals}{\text{argmin}} & \half (u-u_d(x))^2\\
\text{subject to} & L_fb_2(x) + L_gb_2(x) u \geq - 2 b_2(x)\\
& u \in \U
\end{array}
}
where $u_d(x)$ is the desired acceleration. The desired acceleration is computed using the Control Lyapunov Function $V(x) = (x_2 - v_{max})^2$, where $v_{max}=24$ is the speed limit. Thus, $u_d(x)$:
\eqnN{
L_fV(x) + L_gV(x) u_d(x) = -10 V(x)}

The inner safe set $\C^*$ and the input-constrained feedback controller $u^*$ represent the solution to Problem 1 for the ACC problem. 

We compare our controller to the CLF-CBF-QP controller of~\cite{ames2014}:
\eqnN{
\begin{array}{rl}
\underset{u \in \reals, \delta\in \reals_+}{\text{argmin}} & \half u^2 + 0.1 \delta^2\\
\text{subject to} & L_fV(x) + L_gV(x) \leq -10 V(x) + \delta\\
& L_fh(x) + L_gh(x) u \geq - 2h(x)
\end{array}
}
and clip of the solutions of the QP such that $u^*(x)$ lies in the range of feasible control inputs. 

In Figures~\ref{fig:sets_acc}~(e-g), the proposed controller (green) is compared to the CLF-CBF-QP controller (blue), proposed in~\cite{ames2014}. We can see the CLF-CBF-QP reaches the input-constraint at $t=5.9$~seconds. The input limits force the system to leave the safe set. The ICCBF-QP remains feasible and safe for the entire duration, by applying brakes early, at $t=2.9$~seconds, instead of $t=5.0$~seconds. Thus, the ICCBF-QP controller is able to keep the input-constrained system safe, where the CLF-CBF-QP method fails, because it does not consider input constraints.

\subsection{Autonomous Rendezvous}

\begin{figure*}
    \centering
    \includegraphics[width=\linewidth]{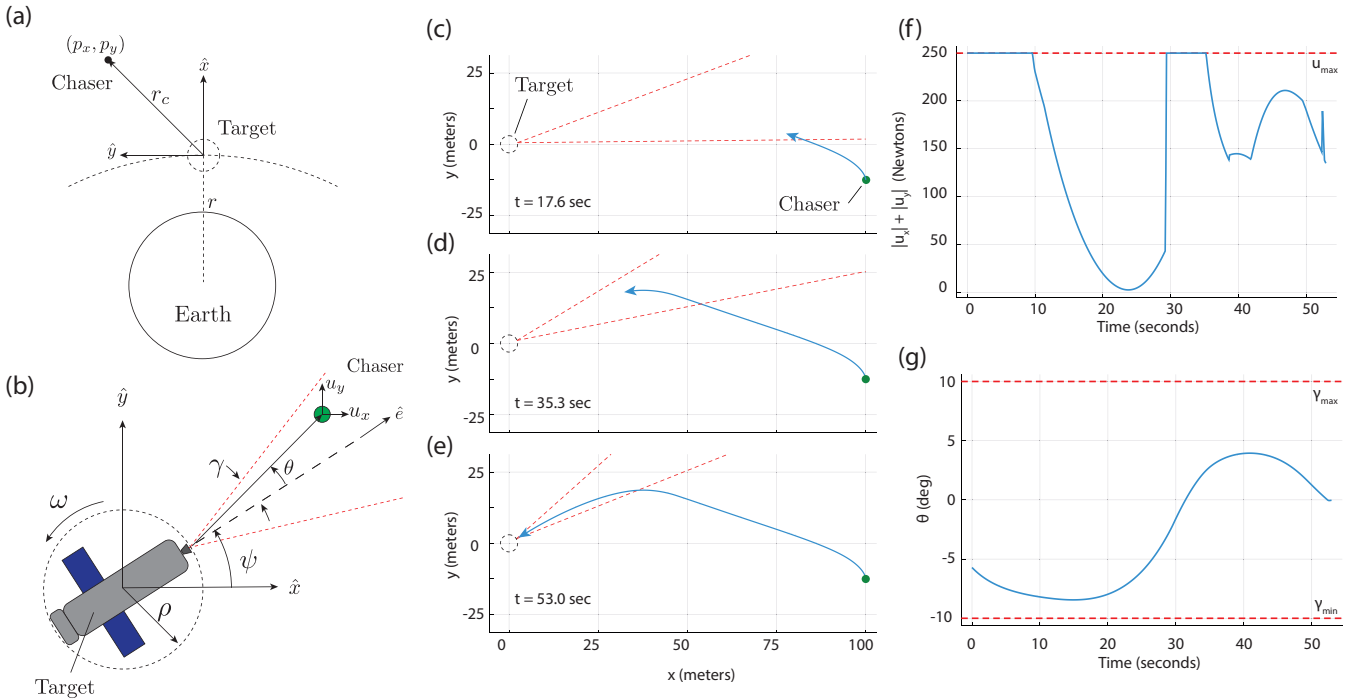}
    \caption{(a, b) Schematic of the rendezvous problem. (a) represents the Local-Vertical Local-Horizontal frame. (b) details the target and chaser spacecrafts. The target spacecraft is rotating with constant angular velocity $\omega$. The red dashed lines indicate the Line-of-Sight cone, which the chaser spacecraft must remain within. (c-e) show snapshots of the trajectory at three instances. The green dot represents the initial condition. (f) shows the 1-norm of the propulsive force and (g) indicates the line of sight angle $\theta$.}
    \label{fig:dock}
\end{figure*}

In this section, the ICCBF method is applied to the case study of an autonomous rendezvous operation between a chaser spacecraft modelled as a point mass, and a target body, for instance the International Space Station (ISS), see Figure~\ref{fig:dock}. The target spacecraft is rotating with a constant angular velocity relative to the Local-Vertical Local-Horizontal (LVLH) frame. The objective is to determine the appropriate propulsive forces to bring the chaser spacecraft from a range of 100~m to 3~m. The safety constraint is  to maintain a line-of-sight (LOS) constraint: the spacecraft's position must remain within a $\gamma = 10^\circ$ cone of the docking axis. This model is adapted from~\cite{docking}, where linearized Model-Predictive-Controllers (LMPC) were used. 

The target is modelled as a disk of radius $\rho=2.4$~m, with a constant angular velocity~$\omega=0.6^\circ$/sec relative to the local-vertical local-horizontal (LVLH) frame, and with a target docking port attached to the edge of the disk. The system state $x \in \reals^5$ is 5 dimensional consisting of the relative position $(p_x, p_y)$ and velocity $(v_x, v_y)$ of the spacecraft in the LVLH frame as the first four states, and angle of the ISS with respect to the LVLH frame $\psi$ as the last state.

Instead of using the linearised Clohessy-Wiltshire equations (as in \cite{docking}), we use the full nonlinear equations of relative motion. In this work, only gravitational forces due to the Earth and propulsive forces are modelled, but other non-linear effects like solar radiation pressure or air resistance can also be included. The dynamical model is then given by
\eqn{
\frac{d}{dt}\bmat{ p_x\\ p_y \\  v_x \\ v_y \\ \psi} = \bmat{
v_x \\
v_y \\
n^2 p_x + 2 n v_y + \frac{\mu}{r^2} - \frac{\mu(r+p_x)}{r_c^3}\\
n^2 p_y - 2 n v_x - \frac{\mu p_y} {r_c^3}\\
\omega
}
+ \frac{1}{m_c}\bmat{0 \\ 0 \\ u_x \\ u_y \\ 0}
}
where $r_c = \sqrt{x^2 + y^2}$ is the relative distance to the chaser, $r=6771$~km is the radius of orbit of the ISS, $\mu=398,600$~km${}^3$/s${}^2$ is the gravitational parameter of Earth, $n = \sqrt{\mu/r^3}$ is the mean motion of the target satellite around the Earth, $\omega=0.6^\circ$/s is the angular velocity of the target relative to the LVLH frame, and $m_c = 1000$~kg is the mass of the chaser vehicle, assumed constant during the rendezvous. The control inputs $(u_x, u_y)$ are the propulsive forces in the local vertical and local horizontal directions, respectively. We suppose that these forces are 1-norm bounded, $|u_x| + |u_y| \leq 0.25$~kN. 

The Line-of-Sight constraint is also non-linear, and can be encoded exactly using $\theta$, the line of sight angle, 
\eqnN{
h(x) &= \cos\theta - \cos\gamma\\
&= \frac{\vec{r}_{c-p} \cdot \hat e}{||\vec{r}_{c-p}||} - \cos(\gamma) ,}
where $\vec r_{c-p} = [(p_x - \rho \cos\psi), (p_y - \rho \sin\psi)]^T$ is the position vector of the chaser relative to the docking port, and $\hat e = [\cos \psi, \sin\psi]^T$ is the direction vector along the docking axis.

The desired behaviour of the spacecraft is encoded using the Control Lyapunov function: 
\eqnN{
V(x) = \left(v_x + \frac{p_x-\rho\cos\psi}{10}\right)^2 + \left(v_y + \frac{p_y - \rho \sin\psi}{10}\right)^2}
which specifies that the desired velocity $v_d$ is along negative $\vec{r}_{c-p}$.

To construct the ICCBF, again $N=2$ was chosen. The following \classK functions were used:
\eqnN{
\alpha_0(h) = 0.25h, \quad \alpha_1(h) = 0.85h, \quad \alpha_2(h) = (0.05 + k)h}
where $k>0$ is a parameter we allow the Quadratic Program to minimize, as in~\cite{kunalCBF}. Thus, the controller $u^*$ is the solution to $u$ in the following quadratic optimization problem
\eqnN{
\begin{array}{rl}
\underset{u \in \reals^2; \delta, k \in \reals_+}{\text{argmin}} & \half (u_x^2 + u_y^2) + 10 \delta + 50 k\\
\text{subject to} &L_fV(x) + L_gV(x) u \leq -0.1 V(x) + \delta\\
&L_fb_2(x) + L_gb_2(x) u \geq - (0.05 + k) b_2(x)\\
& |u_x| + |u_y| \leq 0.25
\end{array}
}

Figure~\ref{fig:dock}(c-g) show simulation results of the rendezvous operation. The chaser is initialized at (100, -10)~meters from the target spacecraft, and follows the trajectories drawn in (c-e), demonstrating a successful transfer. The 1-norm of the computed thrust force is indicated in (f), and (g) shows that the LOS constraint is satisfied at all times during the transfer. 3D animations, videos and source code for both case studies are available at~\cite{me}.

\section{Conclusion}

In this paper, we have presented a framework that allows input constraints to be explicitly included in the construction of control barrier functions and to guarantee that safety is maintained with an input-constrained controller. The construction identifies an inner safe set and a feedback controller to render the subset safe.  The method was demonstrated in two cases, an adaptive cruise control problem and a spacecraft rendezvous control problem. An optimization based method was used to verify the conditions of the ICCBF. Directions for future work include investigating numerically efficient methods to automate the search of ICCBFs, and to compare the complexity with other reachability methods, in particular for systems with high-dimensional states. Finally, the robustness of this controller to noise and model mismatch could also be investigated.


\bibliography{bibfile.bib}{}
\bibliographystyle{IEEEtran}


\end{document}